\documentclass{amsproc}
\usepackage{amsmath}
\usepackage{enumerate}
\usepackage{amsmath,amsthm,amscd,amssymb}

\usepackage{latexsym}
\usepackage{upref}
\usepackage{verbatim}

\usepackage[mathscr]{eucal}
\usepackage{dsfont}

\usepackage{graphicx}
\usepackage[colorlinks,hyperindex,hypertex]{hyperref}

\usepackage{hhline}

\usepackage[OT2,OT1]{fontenc}
\newcommand\cyr
{
\renewcommand\rmdefault{wncyr}
\renewcommand\sfdefault{wncyss}
\renewcommand\encodingdefault{OT2}
\normalfont
\selectfont
}
\DeclareTextFontCommand{\textcyr}{\cyr}

\newtheorem{theorem}{Theorem}
\newtheorem{proposition}[theorem]{Proposition}

\newtheorem{lemma}[theorem]{Lemma}
\newtheorem{definition}[theorem]{Definition}
\newtheorem{remark}[theorem]{Remark}
\newtheorem{hypothesis}[theorem]{Hypothesis}

\chardef\bslash=`\\ 

\newcommand{\fA}{\mathfrak{A}}

\newcommand{\wh}{\widehat}

\newcommand{\dA}{{\dot A}}

\newcommand{\bbR}{{\mathbb{R}}}

\newcommand{\bbC}{{\mathbb{C}}}

\newcommand{\linspan}{\mathrm{lin\ span}}

\newcommand{\ti}{\tilde  }

\newcommand{\dom}{\text{\rm{Dom}}}

\newcommand{\calD}{{\mathcal D}}

\newcommand{\calH}{{\mathcal H}}

\newcommand{\calR}{{\mathcal R}}

\newcommand{\calS}{{\mathcal S}}

\newcommand{\mM}{\mathfrak M}

\newcommand{\whA}{T}
\newcommand{\whB}{T_{\cB}^\kappa}

\renewcommand{\Im}{\text{\rm Im}}

\def\sG{{\mathfrak G}}      
      
\def\sM{{\mathfrak M}}   \def\sN{{\mathfrak N}}   
      
\def\sS{{\mathfrak S}}

\def\bA{{\mathbb A}}      \def\dC{{\mathbb C}}

   \def\cB{{\mathcal B}}   
\def\cD{{\mathcal D}}      
   \def\cH{{\mathcal H}}

      \def\cR{{\mathcal R}}
\def\cS{{\mathcal S}}

\def\RE{{\rm Re\,}}
\def\Ker{{\rm Ker\,}}

\def\wh{\hat}

\def\uphar{{\upharpoonright\,}}

\DeclareMathOperator{\IM}{Im}

\newcommand{\eval}[2][\right]{\relax
  \ifx#1\right\relax \left.\fi#2#1\rvert}

\begin{document}

\title{The L-system representation and c-entropy}

\author{S. Belyi}
\address{Department of Mathematics\\ Troy University\\
Troy, AL 36082, USA\\
}
\curraddr{}
\email{sbelyi@troy.edu}


\author[Makarov]{K. A. Makarov}
\address{Department of Mathematics\\
 University of Missouri\\
  Columbia, MO 63211, USA}
\email{makarovk@missouri.edu}


\author{E. Tsekanovskii}
\address{Department of Mathematics\\ Niagara University, Lewiston, NY
14109\\ USA}
\email{tsekanov@niagara.edu}


\subjclass{Primary 47A10; Secondary 47N50, 81Q10}
\date{DD/MM/2004}

\keywords{L-system, transfer function, impedance function,  Herglotz-Nevan\-linna function, Donoghue class, c-entropy, dissipation coefficient.
\\
  The second author was partially  supported by the Simons collaboration grant 00061759 while preparing this article.}

\begin{abstract}
Given a symmetric operator $\dA$ with deficiency indices $(1,1)$ and its self-adjoint extension $A$ in a Hilbert space $\calH$, we construct a (unique) L-system with the main operator in $\calH$ such that   its  impedance mapping  coincides with the Weyl-Titchmarsh function $M_{(\dot A, A)}(z)$  or  its linear-fractional transformation $M_{(\dot A, A_\alpha)}(z)$. Similar L-system constructions are provided for the  Weyl-Titchmarsh function  $aM_{(\dot A, A)}(z)$ with $a>0$. We also evaluate  $c$-entropy and the main operator dissipation coefficient for the obtained  L-systems. 
 \end{abstract}

\maketitle

\tableofcontents


\section{Introduction}\label{s1}

In this paper we continue our study of the connections between various subclasses of Herglotz-Nevanlinna functions and their conservative realizations as  the impedance functions of L-systems (see \cite{ABT,BMkT,BMkT-2,BMkT-3,Lv2}).

Recall the concept of an L-system.

Let $T$ be a non-symmetric, densely defined linear operator in a Hilbert space $\cH$ such that its resolvent set $\rho(T)$ is not empty. We also assume that
$\dom(T)\cap \dom(T^*)$ is dense and that the restriction $T|_{\dom(T)\cap \dom(T^*)}$ is a closed symmetric operator $\dA$ with deficiency indices $(1,1)$. Let $\calH_+\subset\calH\subset\calH_-$ be the rigged Hilbert space associated with $\dot A$ (see the next section for details).

By   an  \textit{L-system}  we mean the array
\begin{equation}
\label{col0}
 \Theta =
\left(%
\begin{array}{ccc}
  \bA    & K & 1 \\
   \calH_+\subset\calH\subset\calH_- &  & \dC \\
\end{array}%
\right),
\end{equation}
where the \textit{state-space operator} $\bA$ is a bounded linear operator from
$\calH_+$ into $\calH_-$ such that  $\dA \subset T\subset \bA$, $\dA \subset T^* \subset \bA^*$,
$K$ is a bounded linear operator from $\dC$ into $\calH_-$  such that $\IM\bA=KK^*$.

The operator-valued function
\begin{equation*}\label{W1}
 W_\Theta(z)=I-2iK^*(\bA-zI)^{-1}K,\quad z\in \rho(T),
\end{equation*}
 is called the \textit{transfer function}  of the  L-system $\Theta$ and
\begin{equation*}\label{real2}
 V_\Theta(z)=i[W_\Theta(z)+I]^{-1}[W_\Theta(z)-I] =K^*(\RE\bA-zI)^{-1}K,\quad z\in\rho(T)\cap\dC_{\pm},
\end{equation*}
is called the \textit{impedance function } of $\Theta$ (see  Section \ref{s2} for more details).

The main goal of this note can be described as follows.  Given  a symmetric operator $\dA$ with deficiency indices $(1,1)$ and  its arbitrary (but fixed) self-adjoint extension $A$ in a Hilbert space $\calH$,  let $M_{(\dot A, A)}(z)$ denote the   Weyl-Titchmarsh function associated with the pair $(\dot A, A)$. We  construct a unique L-system $\Theta$ of the  form \eqref{col0},  with   $\calH_+\subset\calH\subset\calH_-$   generated by $\dot A$, such that  the  impedance function $V_\Theta(z)$ coincides with $M_{(\dot A, A)}(z)$. Such an L-system we  call the ``\textit{L-system representation}" of the function $M_{(\dot A, A)}(z)$.
 A similar construction is available  for representing the function $M_{(\dot A, A_\alpha)}(z)$ associates with the pair $(\dot A, A_\alpha)$ where $A_\alpha$ is an arbitrary self-adjoint extensions of the symmetric operator  $\dot A$.
 We also  obtain several results that provide us with an explicit construction of an   L-system such that its impedance mapping coincides with  $aM_{(\dA, A)}(z)$ for an arbitrary {value of} the parameter $a>0$. For such  L-systems we  establish an explicit connection between the c-entropy and the coefficient of dissipation of the main operator of the system. In particular, we show that the differential of L-system c-entropy, considered as a one-form  on the Riemann surface $\mathfrak{R}=\bbC\cup\{\infty\}$ of the parameter $a$, is an Abelian differential of the third kind with simple poles at $a=\pm 1$.

 Notice that the main difference in the suggested  approach as compared to the standard techniques used in the realization  theory  for  Herglotz-Nevanlinna functions is that  the underlying triple $\cH_+\subset \cH\subset \cH_-$ in \eqref{col0} is generated by  the symmetric operator $\dot A$ in question.   In order to make this distinction clearer, we use the term ``\textbf{L-system representation}"  instead of the customary ``L-system realization"  when referring  to the corresponding   L-systems.

The paper is organized as follows.

Section \ref{s2} contains necessary information on the L-systems theory.

In Section \ref{s3} we give the formal definition and describe basic properties of the Weyl-Titchmarsh functions   $M_{(\dot A, A)}(z)$ and $M_{(\dot A, A_\alpha)}(z)$.

Section \ref{s4} provides us with an explicit construction of  an L-system associated with   a given  densely defined closed symmetric operator $\dA$ with  deficiency indices $(1,1)$ and fixed $(+)$-normalized deficiency vectors $g_+$ and $g_-$.

Section \ref{s5} contains the main results of the paper, \see Theorems \ref{t-6}  and \ref{t-8}. In particular, we present explicit 
constructions of  L-systems  $\Theta_{10}$ and $\Theta_{\alpha0}$ (see \eqref{e-63-1-0} and \eqref{e-63-alpha-0}) whose impedance functions match $M_{(\dot A, A)}(z)$ and $M_{(\dot A, A_\alpha)}(z)$, respectively. 
 Moreover, we show that the L-systems  $\Theta_{10}$ and $\Theta_{\alpha0}$ have the same main dissipative operator whose von Neumann's parameter with respect to the given  deficiency vectors $g_+$ and $g_-$ (see \eqref{parpar}) is zero.
The obtained results are generalized to a more general setting  where the Weyl-Titchmarsh functions are replaced by the ``re-normalized'' functions $aM_{(\dA,A)}(z)$ and $aM_{(\dA,A_\alpha)}(z)$,  with $a>0$, from  the generalized Donoghue classes discussed  in Section \ref{s3}.

 In Section \ref{s6} we recall the concept of  c-entropy of an L-system introduced  in \cite{BT-16}) and relate  the entropy  to the  dissipation coefficient of  the  main operator of the systems   discussed in  Section \ref{s5}, see Theorems \ref{t-12}--\ref{t-15}.
In particular, we  show that
the c-entropy of the systems associated with the Donoghue functions $aM_{(\dA,A)}(z)$ and $aM_{(\dA,A_\alpha)}(z)$ is infinite if and only if $a=1$.

We  conclude the paper with providing  examples that illustrate   main results.

\section{Preliminaries}\label{s2}

For a pair of Hilbert spaces $\calH_1$ and  $\calH_2$,  denote by
$[\calH_1,\calH_2]$ the set of all bounded linear operators from
$\calH_1$ to $\calH_2$. Let $\dA$ be a closed, densely defined,
symmetric operator in a Hilbert space $\calH$ with inner product
$(f,g)$, $f,g\in\calH$. Throughout this paper, by a  \textit{quasi-self-adjoint extension} of $\dA$ we mean any non-symmetric operator $T$ in $\cH$ such that
\[
\dA\subset T\subset\dA^*.
\]
 Consider the rigged Hilbert space $\calH_+\subset\calH\subset\calH_- $
(see \cite{ABT,Ber,Ber63}), where $$\calH_+ =\dom(\dA^*)$$ and
\begin{equation}\label{108}
(f,g)_+ =(f,g)+(\dA^* f, \dA^*g),\;\;f,g \in \dom(A^*).
\end{equation}

Let $\calR$ be the \textit{\textrm{Riesz-Berezansky   operator}} $\calR$ (see  \cite[Section 2.1]{ABT}, \cite{Ber}, \cite{Ber63}) which maps $\mathcal H_-$ onto $\mathcal H_+$ such
 that   $$(f,g)=(f,\calR g)_+\quad \forall f\in\calH_+, g\in\calH_-$$ with
 $\|\calR g\|_+=\| g\|_-$, so that
\begin{equation}\label{e3-4}
\aligned (f,g)_-=(f,\calR g)=(\calR f,g)=(\calR f,\calR g)_+,\qquad
(f,g\in \mathcal H_-),\\
(u,v)_+=(u,\calR^{-1} v)=(\calR^{-1} u,v)=(\calR^{-1} u,\calR^{-1}
v)_-,\qquad (u,v\in \mathcal H_+).
\endaligned
\end{equation}

 Notice that upon
identifying the space conjugate to $\calH_\pm$ with $\calH_\mp$, we
get that  $\bA\in[\calH_+,\calH_-]$ implies
$\bA^*\in[\calH_+,\calH_-].$

Recall that  {\cite{ABT}}
an operator $\bA\in[\calH_+,\calH_-]$ is called a \textit{self-adjoint
bi-extension} of a symmetric operator $\dot A$ if  $\dA
\subset\bA$ and $\bA^*=\bA$.

Next, if  $\bA$ is  a self-adjoint
bi-extension of $\dA$, the restriction $   \hat A=\bA\uphar\dom(\hat A) $ of  $\bA$
on
$$
\dom(\hat A)=\{f\in\cH_+:\bA f\in\cH\}
$$
 is called a \textit{quasi-kernel} of a self-adjoint bi-extension $\bA$ (see \cite[Section 2.1]{ABT}, \cite{TSh1}).

 A self-adjoint bi-extension $\bA$ of a symmetric operator $\dA$ is called \textit{t-self-adjoint} (see \cite[Definition 4.3.1]{ABT}) if its quasi-kernel $\hat A$ is a self-adjoint operator in $\calH$.
An operator $\bA\in[\calH_+,\calH_-]$  is called a \textit{quasi-self-adjoint bi-extension} of an operator $T$ if $\bA\supset T\supset \dA$ and $\bA^*\supset T^*\supset\dA.$

Also recall that   in
accordance with  the von Neumann Theorem (see \cite[Theorem 1.3.1]{ABT}),
 the domain of an arbitrary self-adjoint extension $\wh A$ of $\dA$  can be represented  as
\begin{equation}\label{DOMHAT}
\dom(\hat A)=\dom(\dA)\oplus(I+U)\sN_{i}.
\end{equation}
Here the operator $U$,  von Neumann's parameter,   is both a $(\cdot)$-isometric as well as $(+)$-isometric operator from the deficiency subspace  $\sN_i$ of $\dA$  into
the deficiency subspace
$\sN_{-i}$
where  $$\sN_{\pm i}=\Ker (\dA^*\mp i I).$$

Throughout this paper, we will be mostly interested in the following type of quasi-self-adjoint bi-extensions.
Let $T$ be a quasi-self-adjoint extension of $\dA$ with nonempty resolvent set $\rho(T)$. A quasi-self-adjoint bi-extension $\bA$ of an operator $T$ is called (see \cite[Definition 3.3.5]{ABT}) a ($*$)\textit{-extension } of $T$ if $\RE\bA$ is a
t-self-adjoint bi-extension of $\dA$.

In what follows we assume that $\dA$ has deficiency indices $(1,1)$. In this case it is known \cite{ABT} that every  quasi-self-adjoint extension $T$ of $\dA$  admits $(*)$-extensions,
and the  description of all $(*)$-extensions via the Riesz-Berezansky   operator $\calR$ can be found in \cite[Section 4.3]{ABT}.

\begin{definition}{[cf. \cite[Definition 6.3.4]{ABT}]}
 An array
\begin{equation}\label{e6-3-2}
\Theta= \begin{pmatrix} \bA&K&\ 1\cr \calH_+ \subset \calH \subset
\calH_-& &\dC\cr \end{pmatrix}
\end{equation}
 is called an \textbf{{L-system}}   if:
\begin{enumerate}
\item[(1)] {$T$ is a dissipative ($\Im(Tf,f)\ge0$, $f\in\dom(T)$) quasi-self-adjoint extension of a symmetric operator $\dA$ with deficiency indices $(1,1)$};
\item[(2)] {$\mathbb  A$ is a   ($\ast $)-extension of  $T$};
\item[(3)] $\IM\bA= KK^*$, where $K\in [\dC,\calH_-]$ and $K^*\in [\calH_+,\dC]$.
\end{enumerate}
\end{definition}
 The operators $T$ and $\bA$ are called the \textit{main and state-space operators respectively} of the system $\Theta$, and $K$ is  the \textit{channel operator}.
It is easy to see that the operator $\bA$ of the system  \eqref{e6-3-2}  can be chosen in such a way  that $$\IM\bA=(\cdot,\chi)\chi,\quad \chi\in\calH_-$$ and
$$K c=c\cdot\chi\quad c\in\dC.$$
 {
Throughout this paper we will often rely on another important element of an L-system, the quasi-kernel $\hat A$ of the $\RE\bA$,  {which is the  self-adjoint extension $\hat A$ of $\dA$  defined as}
\begin{equation}\label{e-5-qk}
    \hat A=\RE\bA\uphar\dom(\hat A),\quad \dom(\hat A)=\{f\in\cH_+:\RE\bA f\in\cH\}.
\end{equation}
}
  The system $\Theta$ in \eqref{e6-3-2} is called \textit{minimal} if the operator $\dA$ is a prime operator in $\calH$, i.e., there exists no non-trivial reducing invariant subspace of $\calH$ on which it induces a self-adjoint operator. Notice that  minimal L-systems of the form \eqref{e6-3-2} with    one-dimensional input-output space were also discussed in \cite{BMkT}.

 With any L-system $\Theta$,  we  associate the  \textbf{transfer  } and  \textbf{impedance function}  given by the formulas
\begin{equation}\label{e6-3-3}
W_\Theta (z)=I-2iK^\ast (\mathbb  A-zI)^{-1}K,\quad z\in \rho (T),
\end{equation}
and
\begin{equation}\label{e6-3-5}
V_\Theta (z) = K^\ast (\RE\bA - zI)^{-1} K, \quad z\in \rho (T),
\end{equation}
respectively.

Recall that  for $\IM z\ne0$, $z\in\rho(T)$ we have
\begin{equation}\label{e6-3-6}
\begin{aligned}
V_\Theta (z) &= i [W_\Theta (z) + I]^{-1} [W_\Theta (z) - I],\\
W_\Theta(z)&=(I+iV_\Theta(z))^{-1}(I-iV_\Theta(z)).
\end{aligned}
\end{equation}
Moreover,
any impedance function $V_\Theta(z)$ admits the  integral representation
\begin{equation}\label{e-60-nu}
V_\Theta(z)=Q+\int_\bbR \left(\frac{1}{\lambda-z}-\frac{\lambda}{1+\lambda^2}\right)d\sigma,
\end{equation}
where $Q$ is a real number and $\sigma$ is an  infinite Borel measure   such that
$$
\int_\bbR\frac{d\sigma(\lambda)}{1+\lambda^2}<\infty.
$$

We refer to  \cite{ABT,BMkT,GT} and references therein for the details and also for the the complete description of the  class of all Herglotz-Nevanlinna functions that can be realized as impedance functions of an L-system.

\section{Donoghue classes and L-systems}\label{s3}

Suppose that $\dA$ is a closed prime 
densely defined symmetric operator with deficiency indices $(1,1)$ in a Hilbert space $\calH$. Assume also that $T \ne T^*$ is a  maximal dissipative extension of $\dot A$,
$$\Im(T f,f)\ge 0, \quad f\in \dom(T ).$$
Since $\dot A$ is symmetric, its dissipative extension $T$ is automatically quasi-self-adjoint \cite{ABT},
that  is,
$$
\dot A \subset T \subset \dA^*,
$$
and hence, (see \cite{BMkT})
\begin{equation}\label{parpar}
g_+-\kappa g_-\in \dom(T)\quad \text{for some } |\kappa|<1,
\end{equation}
where $g_\pm\in\sN_{\pm i}=\Ker (\dA^*\mp i I)$ and  $\|g_\pm\|=1$.
{Throughout this paper  $\kappa$ will be referred to as the \textbf{ von Neumann  parameter} of the operator $T$ with respect to the basis $\{g_\pm\}$ in the subspace $\sN_i\dot +\sN_{-i}.$

Recall that in this setting 
the  Weyl-Titchmarsh  function $M_{(\dot A, A)}(z)$ associated with the pair $(\dot A, A)$ is given by
\begin{equation}\label{e-DWT}
M_{(\dot A, A)}(z)=((Az+I)(A-zI)^{-1}g_+,g_+), \quad z\in \bbC_+.
\end{equation}
Notice that the function $M_{(\dot A, A)}(z)$ is well-defined whenever the normalization condition $\|g_+\|=1$ holds.

Along with the  Weyl-Titchmarsh  function $M_{(\dot A, A)}(z)$, introduce the    {\it  Liv\v{s}ic function} (see \cite{MT-S}, \cite{MTBook})
\begin{equation}\label{charsum}
s(z)=s_{(\dot A, A)}(z)=\frac{z-i}{z+i}\cdot \frac{(g_z, g_-)}{(g_z, g_+)}, \quad
z\in \bbC_+,
\end{equation}
provided that
\begin{equation}\label{unas}
g_+-g_-\in \dom (A).
\end{equation}
Here
$$g_z\in \Ker( \dA^*-zI), \quad z\in \bbC_+.
$$

 Given a  triple  $\fA=(\dot A,  T,A)$,  we call $\kappa$ the {\it von Neumann parameter} of the triple  $\fA=(\dot A,  T,A)$ provided that
the deficiency elements $g_\pm$, $\|g_\pm\|=1$ are chosen in such a way that \eqref{unas} and \eqref{parpar} hold.

In this case, the Weyl-Titchmarsh  $M(z)=M_{(\dot A, A)}(z)$ and  Liv\v{s}ic $s(z)=s_{(\dot A, A)}(z)$ functions  (see \cite {MT-S,MTBook}) are related as
 \begin{equation}\label{blog}
s(z)=\frac{M(z)-i}{M(z)+i},\quad z\in \bbC_+,
\end{equation}
or, equivalently,
\begin{equation}\label{e-12-M-s}
    M(z)=\frac1i\cdot\frac{s(z)+1}{s(z)-1},\quad z\in \bbC_+.
\end{equation}
 We stress that the Weyl-Titchmarsh  $M(z)=M_{(\dot A, A)}(z)$ is rather sensitive to the choice of a self-adjoint reference extension $A$ of the symmetric operator $\dot A$. For instance (see, e.g., \cite[Appendix E, Lemma E.1]{MTBook}),
 if
\begin{equation}\label{dom1}
g_+-e^{2i\alpha}g_-\in \dom (A_\alpha),\quad \alpha \in [0,\pi),
\end{equation}
then the corresponding transformation law for the Weyl-Titchmarsh functions reads as
\begin{equation}\label{transm}
M(\dot A, A_\alpha)=\frac{\cos \alpha \, M(\dot A, A)-\sin \alpha}{
\cos\alpha +\sin \alpha \,M(\dot A, A)},
\end{equation}
while the transformation law for the  corresponding Liv\v{s}ic function is simply reduced to the appearance of an additional ``innocent'' phase factor
$$
s(\dA,A_\alpha)=e^{2i\alpha}s(\dA,A).
$$


Let $\sN$ (see \cite{BMkT-3}) be the class of all Herglotz-Nevanlinna functions $M(z)$ that admit the representation
\begin{equation}\label{hernev-0}
M(z)=\int_\bbR \left(\frac{1}{\lambda-z}-\frac{\lambda}{1+\lambda^2}\right)d\sigma,
\end{equation}
where $\sigma$ is an  infinite Borel measure with
$$
\int_\bbR\frac{d\sigma(\lambda)}{1+\lambda^2}=a<\infty.
$$

 Following our earlier developments in \cite{BMkT,BMkT-3,MT10,MT2021},
 we split   $\sM$ in to three subclasses
 $$
 \sN=\sM_\kappa\cup \sM_0\cup \sM_\kappa^{-1}
 $$
in accordance with whether the norming constant $a$ has the property that
 $a<1$, $a=1$ or $a>1$,
 with
\begin{equation}\label{e-19-kappa}
 \kappa=\begin{cases}\frac{1-a}{1+a}, &a<1\\
 0,&a=1\\
 \frac{a-1}{1+a}, &a>1
 \end{cases}.
\end{equation}
 In particular, $\sM_0=\sM_0^{-1}=\sM$  is the \textbf{Donoghue class} of all analytic mappings $M(z)$ from $\bbC_+$ into itself  that admits the representation
 \eqref{hernev-0}    and has a property
$$
\int_\bbR\frac{d\sigma(\lambda)}{1+\lambda^2}=1\,,\quad\text{equivalently,}\quad M(i)=i.
$$

It is well known  \cite{D,GMT97,GT,MT-S} that $M(z)\in \mM$ if and only if $M(z)$ can be realized  as the Weyl-Titchmarsh function $M_{(\dot A, A)}(z)$ associated with the pair $(\dot A, A)$. Following \cite{BMkT,BMkT-2} we refer to the classes $\sM_\kappa$ and $\sM_\kappa^{-1}$ as the \textbf{generalized Donoghue classes.}

 {We say that an  L-system $\Theta$ of the form \eqref{e6-3-2} is \textit{associated with the triple} $(\dot A, \whA, A)$ if $\dA$ and $T$ are the symmetric and main operators of $\Theta$, respectively, and $A=\hat A$, where $\hat A$ is the quasi-kernel of $\RE\bA$.}
\begin{proposition}[\cite{BMkT, BMkT-2}]\label{prop2}
Suppose that $\whA \ne\whA^*$  is  a maximal dissipative extension of  a symmetric operator $\dot A$  with deficiency indices $(1,1)$ in a Hilbert space $\calH$.  Given $k$, $k\in[0,1)$, assume that the deficiency elements $g_\pm\in \Ker (\dA^*\mp iI)$ are normalized, $\|g_\pm\|=1$, and chosen in such a way that
\begin{equation}\label{ddoomm14}
g_+-\kappa g_-\in \dom (\whA ).
\end{equation}
Suppose that $A$ is a self-adjoint extension of $\dot A$. Assume, in addition, that the Weyl-Titchmarsh function associated with the pair $(\dot A, A)$ is not an identical constant in $\dC_+$.

Denote by  $\Theta$  the minimal L-system of the form \eqref{e6-3-2} associated with the triple $(\dot A, \whA, A)$. Then:

If
\begin{equation}\label{e-21-Hyp}
g_+-g_-\in \dom (A),
\end{equation}
then  the impedance function  $V_{\Theta}$  of the system belongs to the class $\sM_\kappa$.

If, instead,
\begin{equation}\label{e-22-AntHyp}
g_++g_-\in \dom (A),
\end{equation}
then   $V_{\Theta}$ belongs to the class $\sM_{\kappa}^{-1}$.
\end{proposition}
The proof of  Proposition \ref{prop2} can be obtained by combining  \cite[Theorem 12]{BMkT} and \cite[Theorem 5.4]{BMkT-2}.

For the further reference, recall that if the symmetric operator $\dot A $ is prime, then both  the Weyl-Titchmarsh function $M_{(\dot A, A)}(z)$ and  the Liv\v{s}ic function $s_{(\dot A, A)}(z)$ are  complete unitary invariants  of the  pair $(\dot A, A)$  (see  \cite{MT-S,MTBook}).  In particular, the pairs $(\dot A, A)$ and $ (\dot B, B)$ are mutually unitary equivalent if and only if $M_{(\dot A, A)}(z)=M_{(\dot B, B)}(z)$, equivalently, $s_{(\dot A, A)}(z)=s_{(\dot B, B)}(z)$.

\section{Inclusion of  {a dissipative operator} into an L-system}\label{s4}

In  this section, following   \cite{BMkT-3,T69},  we  provide an explicit construction of  an L-system based upon the following operator theoretic setting.

Assume that  $\dA$ is  a densely defined closed symmetric operator with finite deficiency indices $(1,1)$. Given $$
(\kappa,U)\in [0,1)\times \mathbb{T}, \quad \text{with} \quad \mathbb{T}=\{z\in \bbC\,\mid\, |z|=1\},
$$
and $(+)$-normalized deficiency elements $g_\pm\in\sN_{\pm i}=\Ker (\dA^*\mp i I)$,   $\|g_\pm\|_+=1$, assume that $T$ is a quasi-selfadjoint extension of $\dot A$ such that
$$
g_+-\kappa g_-\in \dom (T).
$$
Also assume that $A$ is a reference self-adjoint extension of $\dot A$ with
$$
g_++Ug_-\in \dom (A).
$$

Introduce the  L-system  (see \cite[Appendix A]{BMkT-3})
 \begin{equation}\label{e-215}
\Theta= \begin{pmatrix} \bA&K&\ 1\cr \calH_+ \subset \calH \subset
\calH_-& &\dC\cr \end{pmatrix},
\end{equation}
where $\bA$ is a unique $(*)$-extension $\bA$ of $T$ (see \cite[Theorem 4.4.6]{ABT}) and
$$K\,c=c\cdot\chi,\quad  (c\in\dC).$$

In this case, the state-space  operator $\bA$ given by
\begin{equation}\label{e-205-A}
\begin{aligned}
\bA&=\dA^*+\frac{\sqrt2 i(\kappa+\bar U)}{|1+\kappa U|\sqrt{1-\kappa^2}}\Big(\cdot\,\,, \kappa\varphi+\psi\Big)\chi,
   \end{aligned}
\end{equation}
with
\begin{equation}\label{e-212}
    \chi=\frac{\kappa^2+1+2\kappa U}{\sqrt2|1+\kappa U|\sqrt{1-\kappa^2}}\varphi+ \frac{\kappa^2 U+2\kappa+ U}{\sqrt2|1+\kappa U|\sqrt{1-\kappa^2}}\psi.
\end{equation}
Here
\begin{equation}\label{e-23-phi-psi}
\varphi=\calR^{-1}(g_+),\quad  \psi=\calR^{-1}(g_-),
 \end{equation}
with  $\calR$ the   Riesz-Berezansky   operator.

\begin{remark}\label{r-1}
Notice that since by the hypothesis
$
\|g_\pm\|_+=1,
$
we have
$$\|\varphi\|_-=\|\psi\|_-=1.$$
Indeed, by  \eqref{e3-4},
$$
\|\varphi\|_-^2=\|\cR\varphi\|_+^2=\|g_+\|_+^2=1.
$$
Analogously,
$$
\|\psi\|_-^2=1.
$$
Moreover, since obviously
$$
\|g_\pm\|_+^2=2\|g_\pm\|^2,
$$
we also see that  the deficiency elements  $g_\pm'\in\sN_{\pm i}$ given by
\begin{equation}\label{e-34-conv}
    g_+'=\sqrt2\calR=\sqrt2\, g_+,\qquad g_-'=\sqrt2\calR\psi=\sqrt2\, g_-
\end{equation}
are  $(\cdot)$-normalized.
\end{remark}
 {Given all that, it is also worth mentioning that  for the reminder of this paper all the results will be formulated
in terms of  the $(+)$-normalized deficiency elements $g_\pm$.}

Observe that the constructed $L$-system  $\Theta$ of the form \eqref{e-215} is in one-to-one correspondence with a parametric pair $(\kappa,U)\in [0,1)\times \mathbb T$.

Also recall that   (see \cite{BMkT-3})
\begin{equation}\label{e-26-Im}
        \IM\bA =(\cdot,\chi)\chi,
\end{equation}
and
\begin{equation}\label{e-214}
    \begin{aligned}
    \RE\bA&=\dA^*-\frac{i\sqrt{1-\kappa^2}}{\sqrt2|1+\kappa U|}(\cdot,\varphi-U\psi)\chi,
    \end{aligned}
\end{equation}
where $\chi$ is given by \eqref{e-212}.

If  the reference  self-adjoint extension $A$ satisfies  \eqref{e-21-Hyp} ($U=-1$),
then for the corresponding L-system
\begin{equation}\label{e-62-1-1}
\Theta_1= \begin{pmatrix} \bA_1&K_1&\ 1\cr \calH_+ \subset \calH \subset
\calH_-& &\dC\cr \end{pmatrix}
\end{equation}
  we have
  \begin{equation}\label{e-29-bA1}
    \bA_1=\dA^*-\frac{\sqrt2 i}{\sqrt{1-\kappa^2}}   \Big(\cdot, \kappa\varphi+\psi\Big)\chi_1,
     \end{equation}
where
\begin{equation}\label{e-18}
    \chi_1=\sqrt{\frac{1-\kappa}{2+2\kappa}}\,(\varphi- \psi)=\sqrt{\frac{1-\kappa}{1+\kappa}}\left(\frac{1}{\sqrt2}\,\varphi- \frac{1}{\sqrt2}\,\psi\right).
\end{equation}


  Also, \eqref{e-26-Im} gives us
\begin{equation}\label{e-17}
    \begin{aligned}
    \IM\bA_1&=\left(\frac{1}{2}\right)\frac{1-\kappa}{1+\kappa}(\cdot,\varphi-\psi)(\varphi- \psi)=(\cdot,\chi_1)\chi_1,
       \end{aligned}
\end{equation}
and, according to \eqref{e-214},
\begin{equation}\label{e-17-real}
    \begin{aligned}
    \RE\bA_1&=\dA^*-\frac{i}{2}(\cdot,\varphi+\psi)(\varphi-\psi).
       \end{aligned}
\end{equation}

If  condition \eqref{e-22-AntHyp}   holds, i.e.,  $U=1$,
 the entries of the corresponding
 L-system \begin{equation}\label{e-62-1-3}
\Theta_2= \begin{pmatrix} \bA_2&K_2&\ 1\cr \calH_+ \subset \calH \subset
\calH_-& &\dC\cr \end{pmatrix}
\end{equation}
are given by
\begin{equation}\label{e-29-bA2}
    \bA_2=\dA^*+\frac{\sqrt2 i}{\sqrt{1-\kappa^2}}   \Big(\cdot, \kappa\varphi+\psi\Big)\chi_2,
     \end{equation}
where
\begin{equation}\label{e-18-1}
    \chi_2=\sqrt{\frac{1+\kappa}{2-2\kappa}}\,(\varphi+ \psi)=\sqrt{\frac{1+\kappa}{1-\kappa}}\left(\frac{1}{\sqrt2}\,\varphi+ \frac{1}{\sqrt2}\,\psi\right).
\end{equation}
Also, \eqref{e-26-Im} yields
\begin{equation}\label{e-17-1}
       \IM\bA_2= \left(\frac{1}{2}\right)\frac{1+\kappa}{1-\kappa}\Big((\cdot,\varphi+\psi)(\varphi+\psi)\Big)=(\cdot,\chi_2)\chi_2,
    \end{equation}
and, according to \eqref{e-214},
\begin{equation}\label{e-32-real}
    \begin{aligned}
    \RE\bA_2&=\dA^*-\frac{i}{2}(\cdot,\varphi-\psi)(\varphi+ \psi).
    \end{aligned}
\end{equation}

Note that two L-systems $\Theta_1$ and $\Theta_2$ in \eqref{e-62-1-1} and \eqref{e-62-1-3} are constructed in a way that the quasi-kernels $\hat A_1$ of $\RE\bA_1$ and $\hat A_2$ of $\RE\bA_2$ satisfy the conditions \eqref{e-21-Hyp} or \eqref{e-22-AntHyp}, respectively, as it follows from \eqref{e-17-real} and \eqref{e-32-real}.

 {Concluding this section we would like to emphasize that formulas \eqref{e-215}--\eqref{e-212} allow us to construct an  L-system $\Theta$ that is complectly  based on a given triple $(\dot A, \whA, A)$ and a fixed $(+)$-normalized deficiency vectors $g_\pm$. Moreover, in this construction the operators  $\dA$ and $T$ become the symmetric and main operators of $\Theta$, respectively, while the self-adjoint reference extension  $A$ of the triple matches $\hat A$,  the quasi-kernel of $\RE\bA$.}

\section{Representations of $M_{(\dA,A)}(z)$ and $M_{(\dA,A_\alpha)}(z)$.}\label{s5}


Throughout this section we assume the following Hypothesis.
\begin{hypothesis}\label{hyppo} {
Assume that $\dA$ is a   densely defined closed  symmetric operator  with deficiency indices $(1, 1)$ and $\calH_+\subset\calH\subset\calH_-$  a rigged Hilbert space generated by $\dA$.  Suppose that
deficiency vectors $g_+$ and $g_-$ are $(+)$-normalized. Also assume that  and set
 $\varphi=\calR^{-1}(g_+)$ and $\psi=\calR^{-1}(g_-)$, where $\calR$ is the  Riesz-Berezansky   operator.
 }

\end{hypothesis}

\begin{theorem}\label{t-6}%
Assume Hypothesis \ref{hyppo}. Suppose, in addition, that $A$ is the self-adjoint extensions of $\dot A$ such that
  $$g_+- g_-\in \dom (A).$$

Then the Weyl-Titchmarsh functions  function $M_{(\dA,A)}(z)$ associated with the pair
$(\dot A, A)$ coincides with  the impedance function of  a unique  L-system  of the form
\begin{equation}\label{e-63-1-0}
\Theta_{10}= \begin{pmatrix} \bA_{10}&K_{10}&\ 1\cr \calH_+ \subset \calH \subset
\calH_-& &\dC\cr \end{pmatrix},
\end{equation}
where
\begin{equation}\label{e-20-10}
    \bA_{10}=\dA^*-i(\cdot,\psi)(\varphi-\psi)
\end{equation}
and $$K_{10}c=c\cdot\chi_{10},\quad (c\in\dC)$$ with
\begin{equation}\label{e-18-10}
    \chi_{10}=\frac{1}{\sqrt2}\,(\varphi- \psi).
\end{equation}

In this case, the von Neumann parameter $k$ of the main operator $T_{10}$ of the system  and the von Neumann parameter $U$ of the quasi-kernel $\hat A_{10}=A$
are given by
\begin{equation}\label{kU-1}
\kappa=0\quad \text{and}\quad U=-1.
\end{equation}

Moreover, if  $\Theta$ is  an L-system  \eqref{e6-3-2} with the symmetric operator $\dA$ such that
$$V_{\Theta}(z)=M_{(\dA,A)}(z),\quad z\in \bbC_\pm, $$
where $A=\hat A$ is the quasi-kernel of $\RE\bA$, then the von Neumann  parameters $\kappa$ and $U$  of the
operators $T$ and $\hat A$ of $\Theta$  are given by \eqref{kU-1}.
\end{theorem}
\begin{proof}
We begin by recalling the definition of the model triple \cite {MT-S,MTBook}. Given a function $M(z)$ from the Donoghue class $\sM$ and a $\kappa\in \bbC$, $|\kappa|<1$,
in the Hilbert space $L^2(\bbR;d\sigma)$, where is the representing measure  for the function $M(z)$ via \eqref{hernev-0}, introduce  the    (self-adjoint) operator  $\cB$  of multiplication   by  independent variable  on
\begin{equation}\label{nacha1}
\dom(\cB)=\left \{f\in \,L^2(\bbR;d\sigma) \,\bigg | \, \int_\bbR
\lambda^2 | f(\lambda)|^2d \sigma(\lambda)<\infty \right \}
\end{equation}
and denote by  $\dot \cB$  its symmetric
 restriction
on
\begin{equation}\label{nacha2}
\dom(\dot \cB)=\left \{f\in \dom(\cB)\, \bigg | \, \int_\bbR
f(\lambda)d \sigma(\lambda) =0\right \}.
\end{equation}
Next, introduce
 $\whB$  as   the dissipative quasi-self-adjoint extension  of the  symmetric operator  $\dot \cB$
 on
\begin{equation}\label{nacha3}
\dom(\whB)=\dom (\dot \cB)\dot +\linspan\left
\{\,\frac{1}{\lambda -i}- \kappa\frac{1}{\lambda +i}\right \},
\end{equation}
 Notice that in this case,
\begin{equation}\label{kone}
\dom(\cB)=\dom (\dot \cB)\dot +\linspan\left\{\,\frac{1}{\lambda -i}- \frac{1}{\lambda+i}\right \}.
\end{equation}
We  refer to the triple  $(\dot \cB,   \whB,\cB)$ as {\it  the model  triple }(see
\cite{MT-S,MTBook}) in the Hilbert space $L^2(\bbR;d\sigma)$. It was shown in  \cite[Theorems 1.4, 4.1]{MT-S}
 {that if $\kappa$ is the von Neumann parameter of the dissipative operator $T_A$ with respect to a basis $\{g_\pm\}$, $\|g_+\|=\|g_-\|$, in the deficiency subspaces,
that is,
$$
g_+-\kappa g_-\in \dom (T),
$$
and
$$
g_+-g_-\in \dom (A),
$$
}
then the triple  $(\dot A,T_A,  A)$ is mutually unitarily equivalent to the model triple $(\dot \cB, \whB,  \cB )$ in the  Hilbert space $L^2(\bbR;d\mu)$, where $\mu(d\lambda)$ is the representing measure for the Weyl-Titchmarsh function $M_{(\dA, A)}(z)$.

In \cite[Theorem 7]{BMkT} we have shown how to construct a model L-system that realizes the function $M_{(\dA,A)}(z)$ with the parameters $\kappa=0$ and $U=-1$. The model L-system constructed is uniquely based on the model triple $(\dot \cB, \whB, \cB)$ of the form \eqref{nacha1}--\eqref{nacha3} with $\kappa=0$. This model triple is unitarily equivalent to the triple $(\dA, T_A, A)$ that is included in the L-system $\Theta_1$ in \eqref{e-62-1-1} if we set $\kappa=0$.
Thus, in order to prove the first part of the theorem,  all we have to do is set $\kappa=0$ for  the L-system $\Theta_1$ in \eqref{e-62-1-1} and then refer to \cite[Theorems 1.4, 4.1]{MT-S} to conclude that $M_{(\dA,A)}(z)$ coincides with the impedance function of an L-system with the same set of parameters $\kappa$ and $U$.  We obtain (see \eqref{e-17} -- \eqref{e-17-real}) that
\begin{equation}\label{e-19-10}
    \begin{aligned}
    \IM\bA_{10}&=(\cdot,\chi_{10})\chi_{10},\quad \RE\bA_{10}=\dA^*-\frac{i}{2}(\cdot,\varphi+\psi)(\varphi-\psi),
       \end{aligned}
\end{equation}
and $\chi_{10}$ is defined by \eqref{e-18-10}. Then \eqref{e-20-10}  immediately follows from \eqref{e-19-10}.

Now let us assume that $\Theta$ is   an L-system  with the symmetric operator $\dA$ described in the statement of the theorem  such that $V_{\Theta}(z)=M_{(\dA,A)}(z)$. Then, according to \cite[Theorems 1.4, 4.1]{MT-S}, the underlying triple $(\dA, T, \hat A)$ of $\Theta$ is unitarily equivalent to the model triple $(\dot \cB, \whB, \cB)$ of the form \eqref{nacha1}--\eqref{nacha3} with $\kappa=0$ that is used in the construction of the model L-system realizing $M_{(\dA,A)}(z)$. Consequently, by unitary invariance, both triples $(\dot \cB, \whB, \cB)$ and $(\dA, T, \hat A)$ have the same set of von Neumann's parameters, that is $\kappa=0$ and $U=-1$.
\end{proof}


More generally, we have the following representations theorem.

\begin{theorem}\label{t-8}%
 Assume Hypothesis \ref{hyppo}. Suppose, in addition, that $A$ is the self-adjoint extensions of $\dot A$ such that
 $$g_+-e^{2i\alpha} g_-\in \dom (A_\alpha).$$

  Then  the  Weyl-Titchmarsh function of the pair $(\dA,A_\alpha)$,
coincides with the impedance function of  a unique  L-system of the form
\begin{equation}\label{e-63-alpha-0}
\Theta_{\alpha0}= \begin{pmatrix} \bA_{\alpha0}&K_{\alpha0}&\ 1\cr \calH_+ \subset \calH \subset
\calH_-& &\dC\cr \end{pmatrix},
\end{equation}
where
\begin{equation}\label{e-20-alpha}
    \bA_{\alpha0}=\dA^*-i(\cdot, \psi)(e^{-2i\alpha}\varphi-\psi)
\end{equation}
and
$$K_{\alpha0}c=c\cdot\chi_{\alpha0},\quad (c\in\dC)$$ with
\begin{equation}\label{e-18-alpha}
    \chi_{\alpha0}=\frac{1}{\sqrt2}(\varphi-e^{2i\alpha}\psi).
\end{equation}
In this case, the von Neumann parameter $k$ of the main operator $T_{\alpha0}$ of the system  and the von Neumann parameter $U$ of the quasi-kernel $\hat A_{\alpha0}=A$
are given by
\begin{equation}\label{kUb}
\kappa=0\quad \text{and}\quad  U=
-e^{2i\alpha}.
\end{equation}

Moreover, if  $\Theta$ is  an L-system  \eqref{e6-3-2} with the symmetric operator $\dA$ such that
$$V_{\Theta}(z)=M_{(\dA,A)}(z),\quad z\in \bbC_+, $$
where $A=\hat A$ is the quasi-kernel of $\RE\bA$, then the von Neumann  parameters $\kappa$ and $U$  of the
operators $T$ and $\hat A$ of $\Theta$  are given by \eqref{kUb}.

\end{theorem}
\begin{proof}
Consider $g_-'=(e^{2i\alpha})g_-$. Clearly, $g_-'$ is another $(+)$-normalized deficiency vector in $\Ker( \dA^*+iI)$ and is such that $g_+-g_-'\in \dom (A_\alpha)$. Consequently, all the conditions of Theorem \ref{t-6} apply to the pair of normalized deficiency vectors $g_+$ and $g_-'$ and the  Weyl-Titchmarsh function $M_{(\dA,A_\alpha)}(z)$ of the pair $(\dA,A_\alpha)$. Thus we can construct a realizing $M_{(\dA,A_\alpha)}(z)$ L-system $\Theta_{\alpha0}$ out of L-system $\Theta_{10}$ in \eqref{e-63-1-0} simply by replacing $\psi$ in formulas \eqref{e-20-10}-\eqref{e-18-10} with $\psi'= \calR^{-1}(g_-')$.

Alternatively, we can apply the formulas \eqref{e-212}--\eqref{e-215} with a new set of parameters. In particular, if $U=-e^{2i\alpha}$ and $\kappa=0$ the formulas \eqref{e-212}--\eqref{e-215} become
\begin{equation*}\label{e-212-alpha}
    \chi_{0\alpha}=\frac{1}{\sqrt2}(\varphi-e^{2i\alpha}\psi),
\end{equation*}
and
\begin{equation*}\label{e-214-alpha}
    \begin{aligned}
    \IM\bA_{0\alpha}&=(\cdot,\chi_{0\alpha})\chi_{0\alpha},\quad \RE\bA_{0\alpha}=\dA^*-\frac{i}{\sqrt2}(\cdot,\varphi+e^{2i\alpha}\psi)\chi_{0\alpha},
       \end{aligned}
\end{equation*}
yielding (see also  \eqref{e-205-A})
\begin{equation*}\label{e-205-Alpha}
\begin{aligned}
\bA_{0\alpha}&=\RE\bA_{0\alpha}+i\IM\bA_{0\alpha}=\dA^*-\frac{i}{\sqrt2}\left[(\cdot,\varphi+e^{2i\alpha}\psi)- (\cdot,\varphi-e^{2i\alpha}\psi)\right]\chi_{0\alpha}\\
&=\dA^*-\frac{i}{\sqrt2}(\cdot, 2e^{2i\alpha}\psi)\chi_{0\alpha}=\dA^*-\sqrt2\, ie^{-2i\alpha}(\cdot, \psi)\chi_{0\alpha}\\
&=\dA^*-\sqrt2\, ie^{-2i\alpha}(\cdot, \psi)\frac{1}{\sqrt2}(\varphi-e^{2i\alpha}\psi)=\dA^*-i(\cdot, \psi)(e^{-2i\alpha}\varphi-\psi).
     \end{aligned}
   \end{equation*}
The above formulas confirm \eqref{e-20-alpha} and \eqref{e-18-alpha}.   The corresponding L-system is given by \eqref{e-63-alpha-0} then.

In order to prove the second part of the theorem statement we simply repeat the argument used in the proof of Theorem \ref{t-6} with  $\kappa=0$ and $U=-e^{-2i\alpha}$.
  \end{proof}

Finally, we address a slightly  more general problem of representing  {the functions $aM_{(\dA,A)}(z)$, where
$M_{(\dA,A)}(z)$ is the Weyl-Titchmarsh function associated with the pair $(\dot A, A)$,  with an  arbitrary  $a>0$.}

 Observe that from the definition of the  generalized Donoghue classes  $\sM_\kappa$ and $\sM_\kappa^{-1}$ (see \eqref{e-19-kappa}) it follows that $$aM_{(\dA,A)}(z)\in \sM_\kappa,$$ with
\begin{equation}\label{e-45-kappa-1-new}
\kappa=\frac{1-a}{1+a},
\end{equation}
whenever $0<a<1$
and $$aM_{(\dA,A)}(z)\in \sM_\kappa^{-1},$$ with
\begin{equation}\label{e-45-kappa-2-new}
\kappa=\frac{a-1}{1+a}.
 \end{equation}
 for $a>1$.

\begin{theorem}\label{t-9}%
Assume Hypothesis \ref{hyppo}. Suppose, in addition, that $A$ is the self-adjoint extensions of $\dot A$ such that
  $$g_+- g_-\in \dom (A).$$

 Let  $M_{(\dA,A)}(z)$ denote the  Weyl-Titchmarsh function of the pair $(\dA,A)$.

Then the function $aM_{(\dA,A)}(z)$, ($0<a<1$) coincides with the impedance function of  a unique  L-system  of the form
\begin{equation}\label{e-70-k}
\Theta_{1a}= \begin{pmatrix} \bA_{1a}&K_{1a}&\ 1\cr \calH_+ \subset \calH \subset
\calH_-& &\dC\cr \end{pmatrix},
\end{equation}
where
\begin{equation}\label{e-71-k}
    \bA_{1a}=\dA^*-\frac{i}{2}\left(\cdot,(1-a)\varphi+(1+a)\psi\right)(\varphi-\psi).
\end{equation}
and $$K_{1a}c=c\cdot\chi_{1a}, \quad (c\in\dC)$$ with
\begin{equation}\label{e-72-k}
    \chi_{1a}=\sqrt{\frac{a}{2}}\,(\varphi- \psi).
\end{equation}

In this case, the von Neumann parameter $k$ of the main operator $T_{1a}$ of the system  and the von Neumann parameter $U$ of the quasi-kernel $\hat A_{1a}=A$   are given by
\begin{equation}\label{kU}
\kappa=\frac{1-a}{1+a} \quad \text{and}\quad U=-1.\end{equation}

Moreover, if  $\Theta$ is  an L-system  \eqref{e6-3-2} with the symmetric operator $\dA$ such that
$$V_{\Theta}(z)=aM_{(\dA,A)}(z),\quad z\in \bbC_+, $$
where $A=\hat A$ is the quasi-kernel of $\RE\bA$, then the von Neumann  parameters $\kappa$ and $U$  of the
operators $T$ and $\hat A$ of $\Theta$  are given by  \eqref{kU}.

\end{theorem}
  \begin{proof}
  We follow the steps of the proof of Theorem \ref{t-6} and hence use formulas \eqref{e-17} -- \eqref{e-17-real} with parameters $\kappa$ represented in the form \eqref{e-45-kappa-1-new} and $U=-1$. We get
  \begin{equation}\label{e-72-A}
    \begin{aligned}
    \IM\bA_{1a}&=(\cdot,\chi_{1a})\chi_{1a}={\frac{\sqrt{a}}{\sqrt2}}(\cdot,\varphi- \psi)\chi_{1a},\\ \RE\bA_{1a}&=\dA^*-\frac{i}{2}(\cdot,\varphi+\psi)(\varphi-\psi)=\dA^*-\frac{i}{\sqrt{2a}}(\cdot,\varphi+\psi)\chi_{1a},
       \end{aligned}
\end{equation}
and $\chi_{1a}$ is defined by \eqref{e-72-k}. Then
$$
\begin{aligned}
\bA_{1a}&=\RE\bA_{1a}+i\IM\bA_{1a}=\dA^*-\frac{i}{\sqrt{2a}}\left[(\cdot,\varphi+\psi)-a(\cdot,\varphi-\psi)\right]\chi_{1a}\\
&=\dA^*-\frac{i}{\sqrt{2a}}\left[(\cdot,\varphi+\psi)-a(\cdot,\varphi-\psi)\right]\frac{\sqrt{a}}{\sqrt2}\,(\varphi- \psi)\\
&=\dA^*-\frac{i}{2}\left(\cdot,\varphi-a\varphi+\psi+a\psi\right)(\varphi-\psi)\\
&=\dA^*-\frac{i}{2}\left(\cdot,(1-a)\varphi+(1+a)\psi\right)(\varphi-\psi).
     \end{aligned}
$$
This confirms \eqref{e-71-k}.

In order to prove the second part of the theorem statement we simply repeat the argument used in the proof of Theorem \ref{t-6} with  $\kappa=\frac{1-a}{1+a}$ and $U=-1$.
  \end{proof}
A similar result holds  for $a>1$.

\begin{theorem}\label{t-10}%

Assume Hypothesis \ref{hyppo}. Suppose, in addition, that $A$ is the self-adjoint extensions of $\dot A$ such that
  $$g_++ g_-\in \dom (A).$$

 Let  $M_{(\dA,A)}(z)$ denote the  Weyl-Titchmarsh function of the pair $(\dA,A)$.

Then the function $aM_{(\dA,A)}(z)$, ($a>1$) coincides with the impedance function of  a unique  L-system  of the form
\begin{equation}\label{e-74-k}
\Theta_{1a}= \begin{pmatrix} \bA_{1a}&K_{1a}&\ 1\cr \calH_+ \subset \calH \subset
\calH_-& &\dC\cr \end{pmatrix},
\end{equation}
where
\begin{equation}\label{e-75-k}
    \bA_{1a}=\dA^*-\frac{i}{2}\big(\cdot,(1-a)\varphi-(1+a)\psi\big)(\varphi+\psi).
\end{equation}
and $$K_{1a}c=c\cdot\chi_{1a}, \quad (c\in\dC)$$ with
\begin{equation}\label{e-76-k}
    \chi_{1a}=\sqrt{\frac{a}{2}}\,(\varphi+ \psi).
\end{equation}

In this case, the von Neumann parameter $k$ of the main operator $T_{1a}$ of the system  and the von Neumann parameter $U$ of the quasi-kernel $\hat A_{1a}=A$   are given by
\begin{equation}\label{kU1}
\kappa=\frac{a-1}{1+a} \quad \text{and}\quad U=1.\end{equation}

Moreover, if  $\Theta$ is  an L-system  \eqref{e6-3-2} with the symmetric operator $\dA$ such that
$$V_{\Theta}(z)=aM_{(\dA,A)}(z),\quad z\in \bbC_+, $$
where $A=\hat A$ is the quasi-kernel of $\RE\bA$, then the von Neumann  parameters $\kappa$ and $U$  of the
operators $T$ and $\hat A$ of $\Theta$  are given by  \eqref{kU1}.
\end{theorem}
  \begin{proof}
  We follow the steps of the proof of Theorem \ref{t-6} and hence use formulas \eqref{e-17-1} -- \eqref{e-32-real} with parameters $\kappa$ represented in the form \eqref{e-45-kappa-2-new} and $U=1$. We get
  \begin{equation}\label{e-77-A}
    \begin{aligned}
    \IM\bA_{1a}&=(\cdot,\chi_{1a})\chi_{1a}={\frac{\sqrt{a}}{\sqrt2}}(\cdot,\varphi+ \psi)\chi_{1a},\\ \RE\bA_{1a}&=\dA^*-\frac{i}{2}(\cdot,\varphi-\psi)(\varphi+ \psi)=\dA^*-\frac{i}{\sqrt{2a}}(\cdot,\varphi-\psi)\chi_{1a},
       \end{aligned}
\end{equation}
and $\chi_{1a}$ is defined by \eqref{e-76-k}. Then
$$
\begin{aligned}
\bA_{1a}&=\RE\bA_{1a}+i\IM\bA_{1a}=\dA^*-\frac{i}{\sqrt{2a}}\left[(\cdot,\varphi-\psi)-a(\cdot,\varphi+\psi)\right]\chi_{1a}\\
&=\dA^*-\frac{i}{\sqrt{2a}}\left[(\cdot,\varphi-\psi)-a(\cdot,\varphi+\psi)\right]\frac{\sqrt{a}}{\sqrt2}\,(\varphi+ \psi)\\
&=\dA^*-\frac{i}{2}\left(\cdot,\varphi-a\varphi-\psi-a\psi\right)(\varphi+\psi)\\
&=\dA^*-\frac{i}{2}\left(\cdot,(1-a)\varphi-(1+a)\psi\right)(\varphi+\psi).
     \end{aligned}
$$
This confirms \eqref{e-75-k}.

In order to prove the second part of the theorem statement we simply repeat the argument used in the proof of Theorem \ref{t-6} with  $\kappa=\frac{a-1}{1+a}$ and $U=1$.
  \end{proof}

\section{c-Entropy and  dissipation coefficient }\label{s6}

We begin with reminding a definition of the c-entropy of an L-system introduced in \cite{BT-16}.
\begin{definition}\label{d-9}
Let $\Theta$ be an L-system of the form \eqref{e6-3-2}. The quantity
\begin{equation}\label{e-80-entropy-def}
    \calS=-\ln (|W_\Theta(-i)|),
\end{equation}
where $W_\Theta(z)$ is the transfer function of $\Theta$, is called the \textbf{coupling entropy} (or \textbf{c-entropy}) of the L-system $\Theta$.
\end{definition}

Notice that if, in addition,  the point $z=i$ belongs to $\rho(T)$, then we also have that
\begin{equation}\label{e-80-entropy}
     \calS=\ln (|W_\Theta(i)|)=\ln (1/|\kappa|)=-\ln(|\kappa|).
\end{equation}
This follows from the known (see \cite{ABT}) property of the transfer functions for L-systems  that states that $W_\Theta(z)\overline{W_\Theta(\bar z)}=1$ and  the fact that $|W_\Theta(i)|=1/|\kappa|$ (see \cite{BMkT}).

\begin{definition}\label{d-10}
Let $T$ be the main operator of an L-system $\Theta$  of the form \eqref{e6-3-2} and $\kappa$ be its von Neumman's parameter according to a fixed  $(\cdot)$-normalized deficiency basis $g'_\pm$ such that $0\le\kappa\le1$. If \begin{equation}\label{e-76-ty}
\ti y=g'_+-\kappa g'_-,
\end{equation}
then the quantity $\calD= \IM (T \ti y,\ti y)$ is called the \textbf{coefficient of dissipation}  of the L-system $\Theta$.
\end{definition}


Our next immediate goal is to establish the connection between the coefficient of dissipation $\calD$ and c-entropy of an L-system $\Theta$.
\begin{lemma}\label{l-11}
Let $\Theta$ be an L-system of the form \eqref{e6-3-2}.

Then   c-entropy $\calS$ and the coefficient of dissipation $\calD$ of the system are related as
\begin{equation}\label{e-69-ent-dis}
\calD=1-e^{-2\cS}.
\end{equation}
\end{lemma}
 \begin{proof}
Let $T$ be the main operator of $\Theta$. Since
$$
\ti y=g_+'-\kappa g_-'\in \dom(T)
$$
and $T\subset \dot A^*$, where $\dot A$ is the symmetric operator associated with the system $\Theta$, we have
\begin{align*}
(T\ti y, \ti y)&=(T(g_+'-\kappa g_-'),g_+'-\kappa g_-')=i(g_+'+\kappa g_-',g_+'-\kappa g_-')
\\&=i(\|g_+'\|^2-\kappa^2\|g_-'\|^2)+i\kappa[(g_-',g_+')-(g_+',g_-')]\\
&=i(1-\kappa^2)+i\kappa[(g_-',g_+')-\overline{(g_-',g_+')}].
\end{align*}
Here we have taken into account the normalization condition $\|g_\pm'\|=1$. Note that the second term in the last sum is a real number. Therefore,
\begin{equation}\label{dd}
\cD=\Im (T\ti y, \ti y)=1-\kappa^2.
\end{equation}

Since $|W_\Theta(-i)|=|\kappa|$ (see  \cite{BMkT-2}),
\begin{equation}\label{e-70-entropy}
    \calS=-\ln (|W_\Theta(-i)|)=-\ln(|\kappa|).
\end{equation}

Combining \eqref{dd} and and  \eqref{e-70-entropy} we obtain  the connection between the L-system's  c-entropy and its coefficient of dissipation given by \eqref{e-69-ent-dis}.

\end{proof}

\begin{remark}\label{r-12}
Note that c-entropy defined by formula \eqref{e-70-entropy}  is an additive function with respect to the coupling of L-systems (see \cite{BMkT-2}). That means if L-system $\Theta$ with c-entropy $\calS$ is a coupling of two L-systems $\Theta_1$ and $\Theta_2$ with c-entropies $\calS_1$ and $\calS_2$, respectively, we have
$$
\calS=\calS_1+\calS_2.
$$
Let $\calD$, $\calD_1$, and $\calD_2$ be the dissipation coefficients of L-systems $\Theta$, $\Theta_1$, and $\Theta_2$. Then Lemma \ref{l-11} and \eqref{e-69-ent-dis} imply
$$
1-\calD=e^{-2\cS}=e^{-2(\cS_1+\cS_2)}=e^{-2\cS_1}\cdot e^{-2\cS_2}=(1-\calD_1)(1-\calD_2).
$$
Thus, the formula
\begin{equation}\label{e-72-coupling}
    \calD=1-(1-\calD_1)(1-\calD_2)=\calD_1+\calD_2-\calD_1\calD_2,
\end{equation}
describes the coefficient of dissipation of the L-system coupling.
\end{remark}

Now we are going to find c-entropy for the L-systems whose impedances are the functions $M_{(\dot A, A)}(z)$ and  $aM_{(\dot A, A)}(z)$ with $a>0$.
\begin{theorem}\label{t-12}%

Let $\dA$ be a  symmetric densely defined closed operator  with deficiency indices $(1, 1)$ and  $(+)$-normalized deficiency vectors $g_+$ and $g_-$.

Assume, in addition,  that $A$ is a self-adjoint extensions of $A$
such that
$$g_+- g_-\in \dom (A).$$
Let  $M_{(\dA,A)}(z)$ denote the  Weyl-Titchmarsh function of the pair $(\dA,A)$.

 Then the function $M_{(\dA,A)}(z)$ can be represented as the impedance function of  an   L-system $\Theta_{10}$  of the form \eqref{e-63-1-0} with c-entropy
$\cS$  and   the coefficient of dissipation of the main operator of  the system $\Theta_{10}$ given by
\begin{equation}\label{e-85-ent}
\calS=\infty
\end{equation}
and \begin{equation}\label{e-86-dis1}
\calD=1.
\end{equation}

Moreover,  if $\Theta$ is  an L-system \eqref{e6-3-2} with the symmetric operator $\dA$ such that
$$V_{\Theta}(z)=M_{(\dA,A)}(z),\quad z\in \bbC_+,$$
where $A=\hat A$ is the quasi-kernel of $\RE\bA$,
then  $\Theta$ has infinite c-entropy and the coefficient of dissipation of the main operator is $1$.
\end{theorem}
\begin{proof}
Applying Theorem \ref{t-6} yields that one of the defining $\Theta_{10}$ parameters $\kappa=0$. The proof then immediately follows from \eqref{e-80-entropy} and \eqref{dd}.
\end{proof}
Clearly, a similar to Theorem \ref{t-12} result takes place if the Weyl-Titchmarsh function $M_{(\dA,A)}(z)$ in the statement of the theorem is replaced with $M_{(\dA,A_\alpha)}(z)$. In this case one would apply Theorem \ref{t-8}.

Now we state and prove analogues results for the functions $aM_{(\dA,A)}(z)$ with $a>0$, $a\ne1$.
\begin{theorem}\label{t-14}%
Let $\dA$ be a  symmetric densely defined closed operator  with deficiency indices $(1, 1)$ and  $(+)$-normalized deficiency vectors $g_+$ and $g_-$.

Assume, in addition,  that $A$ is a self-adjoint extensions of $A$
such that
$$g_+- g_-\in \dom (A).$$
Let  $M_{(\dA,A)}(z)$ denote the  Weyl-Titchmarsh function of the pair $(\dA,A)$.

Then for any $a\in (0,1)$, the function $aM_{(\dA,A)}(z)$ can be represented as the impedance function of  an   L-system $\Theta_{1a}$  of the form \eqref{e-70-k} with c-entropy
$\cS$  and   the coefficient of dissipation of the main operator of the system $\Theta_{1a}$ given by
\begin{equation}\label{e-74-ent}
\calS=\ln(1+a)-\ln(1-a)
\end{equation}
and \begin{equation}\label{e-86-dis}
\calD=\frac{4a}{(1+a)^2},
\end{equation}
respectively.

Moreover,  if $\Theta$ is  an L-system \eqref{e6-3-2} with the symmetric operator $\dA$ such that
$$V_{\Theta}(z)=aM_{(\dA,A)}(z),\quad z\in \bbC_+,$$
where $A=\hat A$ is the quasi-kernel of $\RE\bA$, then $\Theta$ has c-entropy $\calS$ and the coefficient of dissipation of the main operator $\calD$ defined by \eqref{e-74-ent} and \eqref{e-86-dis}, respectively.
\end{theorem}
\begin{proof}
Applying Theorem \ref{t-9} yields that one of the defining $\Theta_{1a}$ parameters $\kappa=\frac{1-a}{1+a}$. Substituting this value in \eqref{e-80-entropy} gives us \eqref{e-74-ent}. Using \eqref{dd} we get
$$
\calD=1-\kappa^2=1-\frac{(1-a)^2}{(1+a)^2}=\frac{4a}{(1+a)^2},
$$
that confirms \eqref{e-86-dis}.
\end{proof}
A similar result takes place for $a>1$.
\begin{theorem}\label{t-15}%
Let $\dA$ be a  symmetric densely defined closed operator  with deficiency indices $(1, 1)$ and  $(+)$-normalized deficiency vectors $g_+$ and $g_-$.

Assume, in addition,  that $A$ is a self-adjoint extensions of $A$
such that
$$g_++ g_-\in \dom (A).$$
Let  $M_{(\dA,A)}(z)$ denote the  Weyl-Titchmarsh function of the pair $(\dA,A)$.

Then for any $a>1$ the function $aM_{(\dA,A)}(z)$ can be represented as the impedance function of  an   L-system $\Theta_{1a}$  of the form \eqref{e-74-k} with c-entropy $\cS$  and the   coefficient of dissipation  $\calD$ of the main operator of the system $\Theta_{1a}$ given by
\begin{equation}\label{e-88-ent}
\calS=\ln(a+1)-\ln(a-1)
\end{equation}
and
\begin{equation}\label{cald}
\calD=\frac{4a}{(1+a)^2},
\end{equation}
respectfully.

Moreover,  if $\Theta$ is  an L-system \eqref{e6-3-2} with the symmetric operator $\dA$ such that
$$V_{\Theta}(z)=aM_{(\dA,A)}(z),\quad z\in \bbC_+,$$
where $A=\hat A$ is the quasi-kernel of  $\RE\bA$,
then $\Theta$ has c-entropy $\calS$ and the coefficient of dissipation of the main operator $\calD$ given by \eqref{e-88-ent} and \eqref{e-86-dis}, respectively.
\end{theorem}
\begin{proof}
Applying Theorem \ref{t-10} yields that one of the defining $\Theta_{1a}$ parameters $\kappa=\frac{a-1}{1+a}$. Substituting this value in \eqref{e-80-entropy} gives us \eqref{e-88-ent}. Using \eqref{dd} we get
$$
\calD=1-\kappa^2=1-\frac{(a-1)^2}{(1+a)^2}=\frac{4a}{(1+a)^2},
$$
that confirms \eqref{cald}.
\end{proof}

 {Notice  that c-entropy $\cS(a)$  given by \eqref{e-74-ent} and \eqref{e-88-ent}
admits the representation
$$
\cS(a)=\ln \Big | \frac{a+1}{a-1}\Big |,\quad a\in (0,1)\cup(1,\infty).
$$
It is clearly seen that $\cS(a)$ coincides   with the boundary values on the positive semi-axis of the real part of the analytic function $\sS(a)$ in the upper half-plane
of the complex parameter $a$ given by
$$
\sS(z)=\ln \frac{a+1}{a-1}=\ln(a+1)-\ln (a-1),
\quad a\in \bbC_+.
$$
Here $\ln(z)$  denotes the principle branch of the logarithmic function that is real on the positive real-axis. That is,
$$\cS(a)=\text{Re} (\sG(a+i0)), \quad a\in (0,1)\cap (1,\infty).$$
The point $a=1$ is a branching point of $\sS(a)$: c-entropy
is infinite in this case (see eq. \eqref{e-85-ent}).
}

 {It is interesting to note that the corresponding one-form $d\sS(a)$ is an Abelian differential of the third kind with simple poles at $a=\pm 1$ with
$$\text{Res} (d\sS)|_{a=\pm 1}=\mp1$$ on the Riemann surface $\mathfrak{R}=\bbC\cup\{\infty\}$.
$$
$$
Observe that  the differential  $d\sS(a)$ is invariant with respect to  the involution of the Riemann surface $\mathfrak{R}$ given by
$$
a\mapsto \frac1a, \quad a\in \mathfrak{R}=\bbC\cup\{\infty\}.
$$
In particular,
$$
\cS(a)=\cS\left (\frac1a\right ),  \quad a\in (0,1)\cap (1,\infty).
$$
and hence, by Lemma \ref{l-11},
$$
\cD(a)=\cD\left (\frac1a\right ), \quad a\in (0,\infty).
$$
One cam also see that directly. Indeed, since
$$\cD(a)=\frac{4a}{(1+a)^2}\quad \text{for all}\quad a\in(0,\infty)
$$
(see   \eqref{e-86-dis1}, \eqref{e-86-dis}  and \eqref{cald}),
one computes
$$
\cD(a)=\frac{4a}{(1+a)^2}=\frac{\frac4a}{(1+\frac1a)^2}=\cD\left (\frac1a\right ),\quad a\in (0,\infty).
$$
}


A natural question arises then: can different values $0<a_1<1$ and $a_2>1$ give the same dissipation coefficient and under what condition? To answer it we set
$$
\calD(a_1)=\frac{4a_1}{(1+a_1)^2}=\calD(a_2)=\frac{4a_2}{(1+a_2)^2},\quad 0<a_1<1<a_2,
$$
to get
$$
a_1(1+a_2)^2=a_2(1+a_1)^2.
$$

Expanding and simplifying the above yields
$$
(a_1-a_2)(1-a_1a_2)=0,\quad 0<a_1<1<a_2.
$$
Therefore the only possibility for the dissipation coefficient to stay the same is when $a_2=1/a_1$.
\begin{table}[ht]
\centering
\begin{tabular}{|c|c|c|c|}
\hline
 &  &  &\\
 \textbf{Function}& \textbf{c-Entropy}  & \textbf{Dissipation}  & \textbf{Theorem}  \\
  &  & \textbf{coefficient} &\\ \hline
&  &  &\\
  $M_{(\dA,A)}(z)$ & $\infty$ & $1$ &Theorem \ref{t-12}\\
  &  &  &\\  \hline
 &  &  &\\
 $aM_{(\dA,A)}(z)$&  $\calS=\ln(1+a)-\ln(1-a)$& $\calD=\frac{4a}{(1+a)^2}$ &Theorem \ref{t-14}\\
 &  &  &\\
   $0<a<1$&  & &\\
  \hline
 &  &  &\\
 $aM_{(\dA,A)}(z)$  &  $\calS=\ln(a+1)-\ln(a-1)$& $\calD=\frac{4a}{(1+a)^2}$ &Theorem \ref{t-15}\\
  &  &  &\\
$a>1$   &  & & \\
     \hline
     \multicolumn{1}{l}{} & \multicolumn{1}{l}{} & \multicolumn{1}{l}{} & \multicolumn{1}{l}{}
\end{tabular}
\caption{c-Entropy and Dissipation coefficient}
\label{Table-1}
\end{table}

The results of Section \ref{s6} are summarized in  Table \ref{Table-1} where each row shows a function to be represented  as well as the values of c-Entropy and dissipation coefficient of the   representing  L-system. The last column of Table \ref{Table-1} references theorems stating the corresponding result.

\section{Examples}

\subsection*{Example 1}\label{ex-1}
In this example we will apply Theorem \ref{t-6} to present an L-system that has the impedance function matching a   given function $M_{(\dot A, A)}(z)$.

Following \cite{AG93} we consider the prime symmetric operator for $\ell>0$
\begin{equation}\label{e-56}
    \begin{aligned}
\dA x&=i\frac{dx}{dt},\;
\dom(\dA)=\left\{x(t)\,\Big|\,x(t) -\text{abs. cont.}, x'(t)\in L^2_{[0,\ell]},\, x(0)=x(\ell)=0\right\}.\\
    \end{aligned}
\end{equation}
Its  (normalized) deficiency vectors of $\dA$ are
\begin{equation}\label{e-57}
g_+=\frac{\sqrt2}{\sqrt{e^{2\ell}-1}}e^t\in \sN_i,\qquad g_-=\frac{\sqrt2}{\sqrt{1-e^{-2\ell}}}e^{-t}\in \sN_{-i}.
\end{equation}
If we set $C=\frac{\sqrt2}{\sqrt{e^{2\ell}-1}}$, then \eqref{e-57} can be re-written as
$$
g_+=Ce^t,\quad g_-=Ce^\ell e^{-t}.
$$
Let
\begin{equation}\label{e-58'}
    \begin{aligned}
A x&=i\frac{dx}{dt},\;
\dom(A)=\left\{x(t)\,\Big|\,x(t) -\text{abs. cont.}, x'(t)\in L^2_{[0,\ell]},\, x(0)=-x(\ell)\right\}.\\
    \end{aligned}
\end{equation}
be a self-adjoint extension of $\dA$. Clearly, $g_+(0)-g_-(0)=C-Ce^\ell$ and $g_+(\ell)-g_-(\ell)=Ce^\ell-C$ and hence the first part of \eqref{ddoomm14} is satisfied, i.e., $g_+-g_-\in\dom(A)$.
 Then the Liv\v{s}ic  function $s(z)$ in \eqref{charsum} for the pair $(\dA,A)$ has the form  (see \cite{AG93})
\begin{equation}\label{e1-1}
s(z)=\frac{e^\ell-e^{-i\ell z}}{1-e^\ell e^{-i\ell z}}.
\end{equation}
Applying \eqref{e-12-M-s} we get
\begin{equation}\label{e-88-ex}
    M_{(\dot A, A)}(z)=i\frac{e^\ell+1}{e^{\ell}-1}\cdot\frac{e^{-i\ell z}-1}{e^{-i\ell z}+1}.
\end{equation}
We introduce the operator
\begin{equation}\label{e1-3'}
T_0 x=i\frac{dx}{dt},\;
\end{equation}
$$
\dom(T_0)=\left\{x(t)\,\Big|\,x(t)-\text{abs. cont.}, x'(t)\in L^2_{[0,\ell]},\,  x(\ell)=e^{\ell}x(0)\right\}.
$$
It turns out that  $T_0$ is a dissipative extension of $\dA$ parameterized by the von Neumann parameter $\kappa=0$. Indeed, using \eqref{e-57} with \eqref{parpar} again we obtain
\begin{equation}\label{e-70}
x(t)=Ce^{t}-\kappa\, Ce^\ell e^{-t}\in\dom(T_0),\quad x(\ell)=e^{\ell}x(0),
\end{equation}
yielding $\kappa=0$. 
We are going to use the triple $(\dA, T_0, A)$ in the construction of an L-system $\Theta_0$. By the direct check one gets
\begin{equation}\label{e1-3star'}
T_0^* x=i\frac{dx}{dt},
\end{equation}
$$
\dom(T_0^*)=\left\{x(t)\,\Big|\,x(t)-\text{abs. cont.}, x'(t)\in L^2_{[0,\ell]},\,  x(\ell)=e^{-\ell}x(0)\right\}.
$$
Following  \cite[Example 1]{BMkT} we have
\begin{equation}\label{e7-83}
\begin{aligned}
\dA^\ast x&=i\frac{dx}{dt},\;
\dom(\dA^\ast)&=\left\{x(t)\,\Big|\,x(t)-\text{ abs. cont.}, x'(t)\in L^2_{[0,\ell]}\right\}.\\
\end{aligned}
\end{equation}
Then $\calH_+=\dom(\dA^\ast)=W^1_2$ is the Sobolev space with scalar
product
\begin{equation}\label{e-product}
(x,y)_+=\int_0^\ell x(t)\overline{y(t)}\,dt+\int_0^\ell
x'(t)\overline{y'(t)}\,dt.
\end{equation}
 Construct rigged Hilbert space
$W^1_2\subset L^2_{[0,\ell]}\subset (W_2^1)_-$.

In order to illustrate the results of Section \ref{s5} we need to $(+)$-normalize the deficiency vectors $g_\pm$ from \eqref{e-57}. Using Remark \ref{r-1} we obtain a new set
\begin{equation}\label{e-95}
g_+'=\frac{e^t}{\sqrt{e^{2\ell}-1}}\in \sN_i,\qquad g_-'=\frac{e^{-t}}{\sqrt{1-e^{-2\ell}}}\in \sN_{-i}.
\end{equation}
such that $\|g_\pm'\|_+=1$. Now relying on the fact that $(g_+',\varphi)=(g_-',\psi)=1$ and $(g_+',\psi)=(g_-',\varphi)=0$ we find
\begin{equation}\label{e-96}
    \varphi=\frac{1}{\sqrt{e^{2\ell}-1}}[e^\ell\delta(t-\ell)-\delta(t)]  \quad \textrm{and }\quad \psi=\frac{1}{\sqrt{e^{2\ell}-1}}[e^\ell\delta(t)-\delta(t-\ell)].
\end{equation}
Here $\delta(t)$, $\delta(t-\ell)$ are delta-functions and elements of $(W^1_2)_-$ that generate functionals by the formulas $(x,\delta(t))=x(0)$ and $(x,\delta(t-\ell))=x(\ell)$. For further convenience we find
\begin{equation}\label{e-94-phi-psi}
  \varphi-\psi=  \sqrt{\frac{e^\ell+1}{e^\ell-1}}\left[\delta(t-\ell)-\delta(t)\right].
\end{equation}
Applying Theorem \ref{t-6} and formula \eqref{e-20-10} with $\varphi$ and $\psi$ from \eqref{e-96}, we construct operator
$$
\begin{aligned}
\bA_{10} x&=\dA^*x-i(x,\psi)(\varphi-\psi)\\
&=i\frac{dx}{dt}-i\frac{1}{\sqrt{e^{2\ell}-1}}(e^\ell x(0)- x(\ell))\cdot\sqrt{\frac{e^\ell+1}{e^\ell-1}}\left[\delta(t-\ell)-\delta(t)\right]\\
&=i\frac{dx}{dt}+i \frac{x(\ell)-e^\ell x(0)}{e^\ell-1} \left[\delta(t-\ell)-\delta(t)\right].
\end{aligned}
$$
Finding also $\bA_{10}^\ast$ we obtain
\begin{equation}\label{e-73}
\begin{aligned}
\bA_{10} x&=i\frac{dx}{dt}+i \frac{x(\ell)-e^\ell x(0)}{e^\ell-1} \left[\delta(t-\ell)-\delta(t)\right],\\
\bA_{10}^\ast &x=i\frac{dx}{dt}+i \frac{x(0)-e^{\ell} x(\ell)}{e^\ell-1} \left[\delta(t-\ell)-\delta(t)\right],
\end{aligned}
\end{equation}
where $x(t)\in W_2^1$. It is easy to see that
$\bA_{10}\supset T_0\supset \dA$, $\bA_{10}^\ast\supset T_0^\ast\supset \dA,$
and (independently or via \eqref{e-19-10}) we get
$$
\RE\bA_{10} x=i\frac{dx}{dt}-\frac{i }{2}(x(0)+x(\ell))\left[\delta(t-\ell)-\delta(t)\right].
$$
Moreover, $\RE\bA_0$ has its quasi-kernel equal to $A$ in \eqref{e-58'}. Similarly,
$$
\IM\bA_{10} x=\left(\frac{1}{2}\right) \frac{e^\ell+1}{e^\ell-1} (x(\ell)-x(0))\left[\delta(t-\ell)-\delta(t)\right].
$$
Therefore,
$$
\begin{aligned}
\IM\bA_{10}&=\left(\cdot,\sqrt{\frac{e^\ell+1}{2(e^\ell-1)}}\,\left[\delta(t-\ell)-\delta(t)\right]\right)\sqrt{\frac{e^\ell+1}{2(e^\ell-1)}}\,\left[\delta(t-\ell)-\delta(t)\right]\\
&=(\cdot,\chi_{10})\chi_{10},
\end{aligned}
$$
where (see \eqref{e-18-10})
\begin{equation}\label{e-98-ex}
\chi_{10}=\sqrt{\frac{e^\ell+1}{2(e^\ell-1)}}\,[\delta(t-\ell)-\delta(t)]=\frac{1}{\sqrt2}\,(\varphi- \psi).
\end{equation}
Now we can build
\begin{equation*}
\Theta_{10}= 
\begin{pmatrix}
\bA_{10}&K_{10} &1\\
&&\\
W_2^1\subset L^2_{[0,l]}\subset (W^1_2)_- &{ } &\dC
\end{pmatrix},
\end{equation*}
which is a minimal  L-system of the form \eqref{e-63-1-0} described in Theorem \ref{t-6}. Here
$K_{10} c=c\cdot \chi_{10}$, $(c\in \dC)$, $K_{10}^\ast x=(x,\chi_{10})$ and $x(t)\in W^1_2$. It was shown in \cite[Example 2]{BMkT} the
transfer function of $\Theta_{10}$ is
$$
W_{\Theta_{10}}(z)=\frac{e^\ell e^{-i\ell z}-1}{e^\ell-e^{-i\ell z}}.
$$
The corresponding impedance function also independently calculated in \cite[Example 2]{BMkT} (or found by \eqref{e6-3-6}) is
$$
V_{\Theta_{10}}(z)=i\frac{e^\ell+1}{e^{\ell}-1}\cdot\frac{e^{-i\ell z}-1}{e^{-i\ell z}+1}
$$
that matches $M_{(\dot A, A)}(z)$ from \eqref{e-88-ex}. The c-entropy of the L-system $\Theta_{10}$ is infinite and the coefficient of dissipation of $T_0$ in \eqref{e1-3'} is $\calD=1$.

\subsection*{Example 2}\label{ex-2}
In this example we will apply Theorem \ref{t-9} to present an L-system with impedance mapping matching a  function $aM_{(\dot A, A)}(z)$, $(0<a<1)$. Then we will provide a similar construction  illustrating Theorem \ref{t-10}.

  We  rely on the main elements of the construction presented in Example 1 but with some changes. Let $\dA$ and $A$ be still defined by formulas \eqref{e-56} and \eqref{e-58'}, respectively and let $s(z)$ be  the Liv\v{s}ic characteristic function $s(z)$ for the pair $(\dA,A)$ given by \eqref{e1-1}. We introduce the operator
\begin{equation}\label{e1-3}
    \begin{aligned}
T x&=i\frac{dx}{dt},\;
\dom(T)=\left\{x(t)\,\Big|\,x(t)-\text{abs. cont.}, x'(t)\in L^2_{[0,\ell]},\,  x(0)=0\right\}.
    \end{aligned}
\end{equation}
By construction, $T$ is a dissipative extension of $\dA$ parameterized by the von Neumann parameter $\kappa$. To find $\kappa$ we use \eqref{e-57} with \eqref{parpar} to obtain
\begin{equation}\label{e1-4}
x(t)=Ce^{t}-\kappa\, Ce^\ell e^{-t}\in\dom(T),\quad x(0)=0,
\end{equation}
yielding
\begin{equation}\label{e-102}
\kappa=e^{-\ell}.
\end{equation}

Let
\begin{equation}\label{e-102-ex}
    a=\frac{1-\kappa}{1+\kappa}=\frac{1-e^{-\ell}}{1+e^{-\ell}}=\frac{e^{\ell}-1}{e^{\ell}+1}<1
\end{equation}
and introduce
\begin{equation}\label{e-103-ex}
V(z)=aM_{(\dA,A)}(z)=\frac{e^{\ell}-1}{e^{\ell}+1}M_{(\dA,A)}(z),
\end{equation}
where $M_{(\dA,A)}(z)$ is given by \eqref{e-88-ex}. Clearly, $V(z)=aM_{(\dA,A)}(z)\in \sM_\kappa$ with $\kappa=e^{-\ell}$. We are going to utilize Theorem \ref{t-9} to construct an L-system $\Theta_{1a}$ that represents $V(z)$ as its impedance function.

Following the steps of Example 1 we are going to use the triple $(\dA, T, A)$ in the construction of an L-system $\Theta$.
First, we note that by the direct check one gets
\begin{equation}\label{e1-3star}
    \begin{aligned}
T^* x&=i\frac{dx}{dt},\;
\dom(T)=\left\{x(t)\,\Big|\,x(t)-\text{abs. cont.}, x'(t)\in L^2_{[0,\ell]},\,  x(\ell)=0\right\}.
    \end{aligned}
\end{equation}
Let $W^1_2\subset L^2_{[0,\ell]}\subset (W_2^1)_-$ be the rigged Hilbert space explained in Example 1.
Applying Theorem \ref{t-9} and formula \eqref{e-71-k} with $\varphi$ and $\psi$ from \eqref{e-96}, we obtain operators
\begin{equation}\label{e7-84}
\bA_{1a} x=i\frac{dx}{dt}+i x(0)\left[\delta(t)-\delta(t-\ell)\right],\quad
\bA_{1a}^\ast x=i\frac{dx}{dt}+i x(l)\left[\delta(t)-\delta(t-\ell)\right],
\end{equation}
where $x(t)\in W_2^1$.  It is easy to see that
$\bA_{1a}\supset T\supset \dA$, $\bA_{1a}^\ast\supset T^\ast\supset \dA,$
and that
$$
\RE\bA_{1a} x=i\frac{dx}{dt}+\frac{i }{2}(x(0)+x(\ell))\left[\delta(t)-\delta(t-\ell)\right].
$$
Clearly, $\RE\bA$ has its quasi-kernel equal to $A$ in \eqref{e-58'}. Moreover,
$$
\IM\bA_{1a} =\left(\cdot,\frac{1}{\sqrt 2}[\delta(t)-\delta(t-\ell)]\right) \frac{1}{\sqrt 2}[\delta(t)-\delta(t-\ell)]=(\cdot,\chi_{1a})\chi_{1a},
$$
where $\chi_{1a}=\frac{1}{\sqrt 2}[\delta(t)-\delta(t-\ell)]$. Note that $\bA_{1a}$ represents a $(*)$-extension of operator $T$ with one-dimensional imaginary component in $L^2_{[0,\ell]}$ while the description of all such $(*)$-extensions is given in \cite{T67}.

Now we can build
\begin{equation}\label{e6-125}
\Theta_{1a}=
\begin{pmatrix}
\bA_{1a} &K_{1a} &1\\
&&\\
W_2^1\subset L^2_{[0,\ell]}\subset (W^1_2)_- &{ } &\dC
\end{pmatrix},
\end{equation}
that is a minimal L-system with
\begin{equation}\label{e7-62}
\begin{aligned}
K_{1a}c&=c\cdot \chi_{1a}=c\cdot \frac{1}{\sqrt 2}[\delta(t)-\delta(t-l)], \quad (c\in \dC),\\
K_{1a}^\ast x&=(x,\chi_{1a})=\left(x,  \frac{1}{\sqrt
2}[\delta(t)-\delta(t-l)]\right)=\frac{1}{\sqrt
2}[x(0)-x(l)],\\
\end{aligned}
\end{equation}
and $x(t)\in W^1_2$. It was shown in \cite[Example 1]{BMkT} that the
transfer function of $\Theta_{1a}$ is
$$
W_{\Theta_{1a}}(z)=e^{-i\ell z}.
$$
 The corresponding impedance function is found via \eqref{e6-3-6} and is
$$
V_{\Theta_{1a}}(z)=i\frac{e^{-i\ell z}-1}{e^{-i\ell z}+1}
$$
that matches function $V(z)=aM_{(\dot A, A)}(z)$ of the form \eqref{e-103-ex}. Also, one can confirm by direct substitution that
$$
V_{\Theta_{1a}}(i)=i\frac{e^{\ell}-1}{e^{\ell}+1}=ia.
$$
Using formulas \eqref{e-74-ent} and \eqref{e-86-dis} we find that the c-entropy of the L-system $\Theta_{1a}$ is
$$\calS=\ln\left(1+\frac{e^{\ell}-1}{e^{\ell}+1}\right)-\ln\left(1-\frac{e^{\ell}-1}{e^{\ell}+1}\right)=\ln(e^{\ell})=\ell$$
and the coefficient of dissipation of $T$ in \eqref{e1-3} can be found via \eqref{e-86-dis} as
$$\calD=\frac{4a}{(1+a)^2}=\frac{4\frac{e^{\ell}-1}{e^{\ell}+1}}{(1+\frac{e^{\ell}-1}{e^{\ell}+1})^2}=\frac{e^{2\ell}-1}{e^{2\ell}}=1-e^{-2\ell}.$$

We can quickly modify the developed construction to illustrate Theorem \ref{t-10} using the value
$$
\frac{1}{a}=\frac{e^{\ell}+1}{e^{\ell}-1}>1.
$$
Consider the function
\begin{equation}\label{e-105-ex}
\begin{aligned}
-V^{-1}(z)&=\frac{1}{a}M_{(\dA,A_{\frac{\pi}{2}})}(z)=\frac{e^{\ell}+1}{e^{\ell}-1}\cdot i\frac{e^\ell-1}{e^{\ell}+1}\cdot\frac{e^{-i\ell z}+1}{e^{-i\ell z}-1}\\
&=i\frac{e^{-i\ell z}+1}{e^{-i\ell z}-1},
\end{aligned}
\end{equation}
where $V(z)$ is still defined by \eqref{e-103-ex}. It is known \cite{ABT} (see also \cite{BMkT-2}, \cite{BMkT-3}) that $-V^{-1}(z)$ can be represented as the impedance function of an L-system $\Theta_{1\frac{1}{a}}$ that has the same main operator $T$ (of the form \eqref{e1-3}) as $\Theta_{1a}$.
Applying Theorem \ref{t-10} and formula \eqref{e-75-k} with $\varphi$ and $\psi$ from \eqref{e-96}, and noting that
\begin{equation}\label{e-107-phi+psi}
  \varphi+\psi=  \sqrt{\frac{e^\ell-1}{e^\ell+1}}\left[\delta(t-\ell)+\delta(t)\right],
\end{equation}
we obtain operators
\begin{equation}\label{e-ex2-106}
\bA_{1\frac{1}{a}} x=i\frac{dx}{dt}+ix(0)\left[\delta(t)+\delta(t-\ell)\right] ,\quad
\bA_{1\frac{1}{a}}^\ast x=i\frac{dx}{dt}-ix(l)\left[\delta(t)+\delta(t-\ell)\right],
\end{equation}
where $x(t)\in W_2^1$.  It is easy to see that
$\bA_{1\frac{1}{a}}\supset T\supset \dA$, $\bA_{1\frac{1}{a}}^\ast\supset T^\ast\supset \dA,$
and that
$$
\RE\bA_{1\frac{1}{a}} x=i\frac{dx}{dt}+\frac{i }{2}(x(0)-x(\ell))\left[\delta(t)+\delta(t-\ell)\right],
$$
with the quasi-kernel
$$
A_{\frac{\pi}{2}} x=i\frac{dx}{dt},\,
\dom(A_{\frac{\pi}{2}})=\left\{x(t)\,\Big|\,x(t) -\text{abs. cont.}, x'(t)\in L^2_{[0,\ell]},\, x(0)=x(\ell)\right\}.
$$
Moreover,
$$
\IM\bA_{1\frac{1}{a}} =\left(\cdot,\frac{1}{\sqrt 2}[\delta(t)+\delta(t-\ell)]\right) \frac{1}{\sqrt 2}[\delta(t)+\delta(t-\ell)]=(\cdot,\chi_{1\frac{1}{a}})\chi_{1\frac{1}{a}},
$$
where $\chi_{1\frac{1}{a}}=\frac{1}{\sqrt 2}[\delta(t)+\delta(t-\ell)]$.
Now we can construct
\begin{equation}\label{e-ex2-125}
\Theta_{1\frac{1}{a}}=
\begin{pmatrix}
\bA_{1\frac{1}{a}} &K_{1\frac{1}{a}} &1\\
&&\\
W_2^1\subset L^2_{[0,\ell]}\subset (W^1_2)_- &{ } &\dC
\end{pmatrix},
\end{equation}
that is a minimal L-system with
\begin{equation}\label{e-ex2-110}
\begin{aligned}
K_{1\frac{1}{a}}c&=c\cdot \chi_{1\frac{1}{a}}=c\cdot \frac{1}{\sqrt 2}[\delta(t)+\delta(t-l)], \quad (c\in \dC),\\
K_{1\frac{1}{a}}^\ast x&=(x,\chi_{1\frac{1}{a}})=\left(x,  \frac{1}{\sqrt
2}[\delta(t)+\delta(t-l)]\right)=\frac{1}{\sqrt2}[x(0)+x(l)],\\
\end{aligned}
\end{equation}
and $x(t)\in W^1_2$. The transfer function of $\Theta_{1\frac{1}{a}}$ is
$$
W_{\Theta_{1\frac{1}{a}}}(z)=-e^{-i\ell z},
$$
and the corresponding impedance function is found via \eqref{e6-3-6} and is
$$
V_{\Theta_{1\frac{1}{a}}}(z)=i\frac{e^{-i\ell z}+1}{e^{-i\ell z}-1}
$$
that matches $-V^{-1}(z)=\frac{1}{a}M_{(\dA,A_{\frac{\pi}{2}})}(z)$ from \eqref{e-105-ex}. The c-entropy $\calS$ of the L-system $\Theta_{1\frac{1}{a}}$ equals to the one of the L-system $\Theta_{1a}$ since both L-systems share the same main operator, i.e., $\calS=\ell$. To find the coefficient of dissipation we apply \eqref{e-86-dis} independently
$$\calD=\frac{4\frac{1}{a}}{\left(1+\frac{1}{a}\right)^2}=\frac{4\frac{e^{\ell}+1}{e^{\ell}-1}}{(1+\frac{e^{\ell}+1}{e^{\ell}-1})^2}=\frac{e^{2\ell}-1}{e^{2\ell}}=1-e^{-2\ell}.$$
 This confirms our observation made in the end of Section \ref{s6} that is $\calD(1/a)=\calD(a)$.

\begin{remark}\label{r-16}
Note that the symmetric operator $\dA$ of the form \eqref{e-56}, that was used in Examples 1 and 2 for the construction of the representing L-systems, is a model operator (according to the Liv\u sic Theorem  \cite[Sect. IX, Theorem 4]{AG93}, \cite{L}) for any prime symmetric operator with deficiency indices $(1, 1)$ that admits dissipative extension without the spectrum in the open upper half-plane.
\end{remark}



\begin{thebibliography}{1}



\bibitem{AG93}
N. I.~Akhiezer, I. M.~Glazman. \textit{Theory of linear operators}.  Pitman {A}dvanced {P}ublishing {P}rogram, 1981.






\bibitem{ABT}
{Yu.~Arlinskii, S.~Belyi, E.~Tsekanovskii},
\textit{Conservative Realizations  of Herglotz-Nevanlinna functions}, {Oper. Theory Adv. Appl.}, {vol. \textbf{217}}, {Birkh\"auser Verlag}, Basel, {2011}.


\bibitem{BMkT}
 S.~Belyi, K.~A.~ Makarov, E.~Tsekanovskii,  \textit{Conservative  L-systems and the Liv\v{s}ic function}. Methods of Functional Analysis and Topology,  \textbf{21}, no. 2,  (2015),  104--133.

\bibitem{BT-16}
 S.~Belyi, K.~A.~ Makarov, E.~Tsekanovskii,   \textit{On the c-entropy of L-systems with Schrodinger operator},  Complex Analysis and Operator Theory,  \textbf{16} (107), (2022), 1--59.


 \bibitem{BMkT-2}
S.~Belyi,  K.~A.~Makarov, E.~Tsekanovskii, \textit{A system coupling and Donoghue classes of Herglotz-Nevanlinna functions},  Complex Analysis and Operator Theory,  \textbf{10} (4), (2016), 835--880.



 \bibitem{BMkT-3}
S.~Belyi,   E.~Tsekanovskii, \textit{Perturbations of Donoghue classes and inverse problems for L-systems},  Complex Analysis and Operator Theory,  vol. \textbf{13} (3), (2019),  1227--1311.



\bibitem{Ber}
{Yu.~Berezansky},  \textit{Expansion in eigenfunctions of self-adjoint operators}, {vol. \textbf{17}}, {Transl. Math. Monographs}, {AMS}, {Providence}, {1968}.

\bibitem{Ber63}
{Yu.~Berezansky},  \textit{Spaces with negative norm}, (Russian), {Uspehi Mat. Nauk}, { \textbf{18}}, no.1 (109), {(1963)}, 63--96.


\bibitem {D}
 W. F.~Donoghue,
\textit{On perturbation of spectra}. Commun. Pure and Appl. Math., {\bf 18}, (1965),  559--579.



\bibitem{GMT97} F.~Gesztesy, K. A.~Makarov, E.~Tsekanovskii,  \textit{An addendum to Krein's formula}, {J. Math. Anal. Appl.}, {\bf 222}, (1998), 594--606.

\bibitem{GT}
F.~Gesztesy, E.~Tsekanovskii, \textit{On Matrix-Valued Herglotz Functions}. Mathematische Nachrichten, \textbf{218}, (2000), 61--138.



\bibitem{Lv1} M.~Liv\v{s}ic,  \textit{On spectral decomposition of linear non-self-adjoint operators}.  Mat. Sbornik (76) {\bf 34} (1954),  145--198 (Russian);  English transl.:  Amer. Math. Soc. Transl. (2)  {\bf 5},   (1957), 67--114.

\bibitem{L}
M.~Liv\v{s}ic, \textit{On a class of linear operators in Hilbert space}. Mat. Sbornik  (2), {\bf 19} (1946), 239--262 (Russian); English transl.:  Amer. Math. Soc. Transl., (2), {\bf 13}, (1960), 61--83.

\bibitem{Lv2}
M. S.~Liv\v{s}ic, \textit{Operators, oscillations, waves}. Moscow, Nauka, 1966.


\bibitem{MT-S}
{K.~A.~Makarov, E.~Tsekanovskii}, \textit{On the Weyl-Titchmarsh and Liv\v{s}ic  functions}. Proceedings of Symposia in Pure Mathematics, vol. \textbf{87}, American Mathematical Society, (2013), 291--313.


\bibitem{MT10}
{K.~A.~Makarov, E.~Tsekanovskii}, \textit{On the addition and  multiplication theorems}.  	 {Oper. Theory Adv. Appl.}, {vol. \textbf{244}}, (2015), 315--339.

\bibitem{MT2021}
{K.~A.~Makarov, E.~Tsekanovskii}, \textit{Representations of commutation relations in Dissipative Quantum Mechanics}.  	 ArXiv: 2101.10923 [math-ph].

\bibitem{MTBook}
K.~A.~Makarov, E.~Tsekanovskii,
\textit{The Mathematics of Open Quantum Systems}, World Scientific, NJ, 2022.


\bibitem{T69}
 E.~Tsekanovski\u i, \textit{The description and the uniqueness of generalized extensions of quasi-Hermitian operators}. (Russian) Funkcional. Anal. i Prilozen., \textbf{3}, no. 1,  (1969), 95--96.

\bibitem{T67}
E.~Tsekanovski\u i, \textit{Description of the generalized extensions with rank one imaginary component of a differentiation operator without spectrum.} (Russian), Dokl. Akad. Nauk SSSR, \textbf{176}, (1967), 1266--1269.


\bibitem{TSh1}
{E.~Tsekanovskii, Yu. L.~\u Smuljan,}   \textit{The theory of bi-extensions of operators on rigged Hilbert spaces. Unbounded operator colligations and characteristic functions},
 {Russ. Math. Surv.}, {\bf 32},  {(1977)}, {73--131}.

\end{thebibliography}
\end{document}